\documentclass[preprint,11pt, english]{elsarticle}
\usepackage{amsmath}
\usepackage{latexsym, amssymb}
\usepackage{amsthm}
\usepackage{color}
\usepackage{txfonts}
\usepackage{calrsfs}

\usepackage{graphics}
\usepackage{rotating}
\usepackage{comment}

\input xy
\input xyidioms.tex
\usepackage{xy}
\xyoption{all} %

\newtheorem{thm}{Theorem}[section] 

\newtheorem{cor}[thm]{Corollary}
\newtheorem{lem}[thm]{Lemma}
\newtheorem{prop}[thm]{Proposition}

\theoremstyle{definition}
\newtheorem{rem}[thm]{Remark}
\newtheorem{exmpl}[thm]{Example}
\newtheorem{defn}[thm]{Definition}

\newtheorem{notation}[thm]{Notation}

\newcommand\operA[2]{{\if!#2!\operatorname{#1}\else{\operatorname{#1}_{#2}^{\phantom{I}}}\fi}} 

\newcommand\Cref[1]{{Corollary~\ref{#1}}}%

\renewcommand{\phi}{\varphi}
\def\tr{{\operatorname{Tr}}}
\def\dlog{{\operatorname{dlog}}}

\def\Br{{\operatorname{Br}}}

\def\CH{{\operatorname{CH}}}
\def\TCH{{\operatorname{TCH}}}

\def\Z{\mathbb{Z}}
\def\F{\mathbb{F}}
\def\K{\mathbb{K}}

\newcommand{\A}{\mathbb A}
\renewcommand{\phi}{\varphi}

\newcommand{\Trace}[1][]{\if!#1!\operatorname{Tr}\else{\operatorname{Tr}_{#1}^{\phantom{I}}}\fi} 

\long\def\forget#1\forgotten{{}} %

\def\({\left(}
\def\){\right)}

\newif\iffurther
\furtherfalse

\newif\ifXY 
\XYtrue     
\ifXY

\input xy
\input xyidioms.tex
\usepackage{xy}
\xyoption{all} %
\fi 

\usepackage{babel}

\journal{Israel Journal of Mathematics}

\begin{document}

\begin{frontmatter}

\title{The descent of biquaternion algebras in characteristic two}

\author[1,2]{Demba Barry}
\ead{barry.demba@gmail.com}

\author[3]{Adam Chapman}
\ead{adam1chapman@yahoo.com}

\author[4]{Ahmed Laghribi}
\ead{ahmed.laghribi@univ-artois.fr}

\address[1]{Facult\'e des Sciences et Techniques de Bamako, BP: E3206 Bamako, Mali}
\address[2]{Departement Wiskunde--Informatica, Universiteit Antwerpen, Belgium}

\address[3]{Department of Computer Science, Tel-Hai College, Upper Galilee, 12208 Israel}

\address[4]{Universit\'e d'Artois, Facult\'e des Sciences Jean Perrin, Laboratoire de math\'ematiques de Lens EA 2462, rue Jean Souvraz - SP18, 62307 Lens, France}

\begin{abstract}
In this paper we associate an invariant to a biquaternion algebra $B$ over a field $K$ with a subfield $F$ such that $K/F$ is a quadratic separable extension and $\operatorname{char}(F)=2$. We show that this invariant is trivial exactly when $B \cong B_0 \otimes K$ for some biquaternion algebra $B_0$ over $F$.
We also study the behavior of this invariant under certain field extensions and provide several interesting examples.
\end{abstract}


\end{frontmatter}

\section{Introduction}

Given a central simple algebra $C$ over a field $K$ with a subfield $F$, we say that $C$ has a descent to $F$ if there exists a central simple algebra $C_0$ over $F$ such that $C \cong C_0 \otimes K$.
When $[K:F]=\exp(C)$, one necessary condition for $C$ to have a descent to $F$ is that $\operatorname{cor}_{K/F}(C) \sim_{Br} F$, because if $C \cong C_0 \otimes K$ then $\operatorname{cor}_{K/F}(C)=\operatorname{cor}_{K/F}(C_0 \otimes K)=C_0^{\otimes [K:F]} \sim_{Br} F$, where $\sim_{Br}$ denotes the Brauer equivalence.

If $K/F$ is a separable quadratic extension and $Q$ is a quaternion algebra over $K$, then\break $\operatorname{cor}_{K/F}(Q) \sim_{Br} F$ is a necessary and sufficient condition for $Q$ to have a descent to $F$ by \cite[Chapter X, Theorem 21]{Albert:1968}. This fact does not generalize to biquaternion algebras.
For example, take a division algebra $A$ over $F$ such that $\deg(A)=8$ and $\exp(A)=2$ which does not decompose as a tensor product of three quaternion algebras (e.g. \cite{AmitsurRowenTignol:1979}), take a separable quadratic field extension $K$ of $F$ inside $A$ (which exists by \cite{Rowen:1978}) and $B$ to be the centralizer of $K$ in $A$. Then $B \sim_{Br} A \otimes K$ and so $\operatorname{cor}_{K/F}(B) \sim_{Br} F$. However, if $B \cong B_0 \otimes K$ for some central simple algebra $B_0$ over $F$, then $A$ decomposes as the tensor product of three quaternion algebras, a contradiction. Therefore $B$ has no descent to $F$.

In \citep{Barry}, the invariant $\delta$ was associated to any biquaternion algebra $B$ over $K$ where $K/F$ is a quadratic field extension and $\operatorname{char}(F) \neq 2$. It was shown that the invariant is trivial exactly when $B$ has a descent to $F$.
This invariant proved to be a refinement of an invariant $\Delta$ defined in \cite[Section 11]{GaribaldiParimalaTignol} for algebras of degree 8 and exponent 2, which is trivial when the algebra decomposes and is conjectured to be nontrivial otherwise.
It was then used to show the existence of at least one algebra $A$ of degree 8 and exponent 2 with nontrivial $\Delta(A)$.
The main tool Barry used in order to construct this example was \cite[Proposition 4.9]{Barry}, whose characteristic 2 analogue is given in this paper as Proposition \ref{Injection}. 

In this paper we present the characteristic 2 analogue for $\delta$. We associate an invariant to any biquaternion algebra $B$ over $K$ where $K/F$ is a separable quadratic extension and 
$\operatorname{char}(F) =2 $. We show that this invariant is trivial exactly when $B$ has a descent to $F$, and study its behavior under certain field extensions, proving (among other things) that it does not split over odd degree field extensions (see Proposition \ref{injoddext}).
This invariant is defined using the Kato-Milne cohomology groups, which form a characteristic 2 analogue to the classical Galois cohomology groups.

\section{Kato-Milne cohomology}

Throughout this paper $F$ denotes a field of characteristic 2.
Let $\Omega_F^m=\wedge^m\Omega_F^1$ be the space of absolute $m$-differential forms over $F$, where $\Omega_F^1$ is the $F$-vector space generated by the symbols $dx$, $x \in F$, subject to the relations: 
\begin{eqnarray*}
d(x+y)&=&dx+dy,\\
d(xy)&=&xdy+ydx
\end{eqnarray*}
for any $x,y \in F$. We set $\Omega_F^0=F$. Clearly, $d(x^2y)=x^2d(y)$ for all $x,y\in F$, and thus the map 
\begin{eqnarray*}
F &\rightarrow& \Omega_F^1\\
x &\mapsto& dx
\end{eqnarray*}
is $F^2$-linear, where $F^2=\{x^2 \mid x \in F\}$. This map extends to the differential operator
$d: \Omega_F^m  \rightarrow \Omega_F^{m+1}$ defined by
$$y dx_1 \wedge \dots \wedge dx_m \mapsto dy \wedge dx_1 \wedge \dots \wedge dx_m.$$

Let $\mathcal{B}=\{e_i \mid i \in I\}$ be a 2-basis of $F$, which means that the set$$\left\{ \prod_{i \in I} e_i^{\varepsilon_i} \mid \varepsilon_i \in \{0,1\}, \enspace \text{and for almost all} \enspace i \in I, \varepsilon_i=0\right\}$$is a basis of $F$ over $F^2$. We choose an ordering on $I$. So the set 
$$\left \{\frac{de_{i_1}}{e_{i_1}} \wedge \dots \wedge \frac{de_{i_m}}{e_{i_m}} \mid e_{i_1},\dots,e_{i_m} \in \mathcal{B} \enspace \text{and} \enspace i_1 < i_2 <\dots <i_m \right \}$$
is a basis of the $F$-vector space $\Omega_F^m$.

The usual Frobenius map $F\rightarrow F$, $x\mapsto x^2$, extends to a well-defined map, called the Frobenius operator, as follows:
\begin{eqnarray*}
\Phi : \Omega_F^m &\rightarrow & \Omega_F^m/d\Omega_F^{m-1}\\
\sum_{i_1< \dots <i_m} c_i \frac{de_{i_1}}{e_{i_1}} \wedge \dots \wedge \frac{e_{i_m}}{e_{i_m}} &\mapsto & \sum_{i_1< \dots <i_m} c_i^2 \frac{de_{i_1}}{e_{i_1}} \wedge \dots \wedge \frac{e_{i_m}}{e_{i_m}}+d\Omega_F^{m-1}.
\end{eqnarray*}

The Artin-Schreier operator $\wp : \Omega_F^m \rightarrow \Omega_F^m/d\Omega_F^{m-1}$ is defined by $\wp=\Phi-\operatorname{Id}$. 

Let $H_2^{m+1}(F)$ and $\nu_F(m)$ denote the cokernel and the kernel of $\wp$, respectively. Usually we will consider the operator $\wp$ as a map $\Omega_F^m \rightarrow \Omega_F^m$ given by: $$\sum_{i_1< \cdots <i_m} c_i \frac{de_{i_1}}{e_{i_1}} \wedge \cdots \wedge \frac{e_{i_m}}{e_{i_m}} \mapsto  \sum_{i_1< \dots <i_m} \wp(c_i) \frac{de_{i_1}}{e_{i_1}} \wedge \dots \wedge \frac{e_{i_m}}{e_{i_m}}.$$

This map depends on the choice of the $2$-basis, and it is well-defined modulo $d\Omega_F^{m-1}$. With that we may take $$H_2^{m+1}(F)=\Omega_F^m/(d\Omega_F^{m-1}+\wp(\Omega_F^m)).$$ 

A famous result of Kato \cite{Kato:1982} asserts the following isomorphism:
\begin{eqnarray*}
f_{m+1}: H_2^{m+1}(F) & \rightarrow & I^m F \otimes W_q(F)/ I^{m+1} F \otimes W_q(F)\\
\overline{a \frac{dx_1}{x_1} \wedge \cdots \wedge \frac{dx_m}{x_m}} & \mapsto & \overline{\langle \langle x_1,\cdots , x_m \rangle \rangle \otimes [1,a]}
\end{eqnarray*}
where $I^mF$ is the $m$th power of the fundamental ideal of the Witt ring $W(F)$ of nondegenerate symmetric $F$-bilinear forms, $W_q(F)$ is the Witt group of non singular $F$-quadratic forms, $\left<\left< x_1,\cdots , x_n \right>\right>$ is the $n$-fold bilinear Pfister form $\left< 1,x_1 \right> \otimes \cdots \otimes \left< 1,x_n \right>$, and $[1,a]$ is the binary quadratic form $X^2+XY+a Y^2$. We have $H_2^1(F)\cong F/\wp(F)$ and $H_2^2(F)\cong {}_2\Br(F)$ the 
$2$-torsion of the Brauer group of $F$.

Kato also proved in \cite{Kato:1982} that $\nu_F(m)$ coincides with the additive group generated by the differentials $\frac{da_1}{a_1}\wedge \ldots \wedge \frac{da_m}{a_m}$ for $a_1, \cdots, a_m\in F^{\times}$, and it is isomorphic to the quotient $I^mF/I^{m+1}F$ in a natural way.

\section{An invariant for biquaternion algebras over a quadratic extension}\label{subsection2}


From now on $K$ denotes a separable quadratic extension of $F$. We have $K=F[\alpha]$, where $\wp(\alpha)=a \in F \setminus \wp(F)$.
Without loss of generality, we can assume $a \in F^2$. 

Because of the separability of the extension $K/F$, any $2$-basis of $F$ remains a $2$-basis of $K$, and thus we have $\Omega_K^m=\Omega_F^m \oplus \alpha \Omega_F^m$ for any nonnegative integer $m$.

Let $\tr : K \rightarrow F$ denote the trace map. This map extends to 
%
%
\begin{eqnarray*}
\tr_* : \Omega_K^m &\rightarrow& \Omega_F^m\\
w_0+\alpha w_1 & \mapsto & w_1
\end{eqnarray*}
for all $w_0,w_1 \in \Omega_F^m$.

Since $a \in F^2$, we also have $\alpha \in K^2$, and thus$$d\Omega_K^m = d\Omega_F^m \oplus \alpha d\Omega_F^m$$$$\wp(w_0+\alpha w_1)=\wp(w_0)+a \Phi(w_1)+\alpha \wp(w_1)$$for all $w_0,w_1 \in \Omega_F^m$.
Hence, $\tr_*(d \Omega_K^m) \subset d \Omega_F^m$ and $\tr_*(\wp(\Omega_K^m)) \subset \wp(\Omega_F^m)$.
Consequently, $\tr_*$ extends to a map from $H_2^{m+1}(K)$ to $H_2^{m+1}(F)$.

%

%

Moreover, by \citep[Corollary 2.5]{AravireBaeza:1989} (see also \cite[Complex 34.20]{EKM}), the group homomorphism $\tr_*: W_q(K) \longrightarrow W_q(F)$ induces by the trace map $\tr : K \rightarrow F$ satisfies $\tr_*(I^n K \otimes W_q(K)) \subset I^n F \otimes W_q(F)$ for all $n \geq 0$. This is a part of the following exact sequence:

\begin{prop}[{\cite[Cor. 6.5]{AravireBaeza:1992}}]
The following sequence is exact:
$$\xymatrix{ 0 \ar[r]& I^nF \otimes [1,a] \ar[r] & I^nF \otimes W_q(F) \ar@{->}[r]^{\operatorname{res}_{K/F}} & I^nK \otimes W_q(K) \ar@{->}[r]^{\tr_*} & I^nF \otimes W_q(F) \ar[r] &0
}$$
where $\operatorname{res}_{K/F}$ is the restriction map.
\label{exaseq}
\end{prop}

We will also need the following result describing the group $I^mK \otimes W_q(K)$.

\begin{prop}[{\cite[Corollary 2.4]{AravireBaeza:1989}}]\label{Generators}
For any $n \geq 0$, $I^n K \otimes W_q(K)$ is generated by the Pfister forms $\langle\langle a_1,\dots,a_n,b]]$ where $a_1,\dots,a_n \in F^\times$ and $b \in K$.
\end{prop}

For any $m\geq 0$, let us denote by $e^{m+1}:I^mF \otimes W_q(F)\rightarrow H_2^{m+1}(F)$ the group homomorphism  given by:$$\varphi \mapsto f_{m+1}^{-1}(\varphi+ I^{m+1}F \otimes W_q(F))$$where $f_{m+1}$ is the Kato isomorphism given in the previous section. In particular, for any $\phi \in I^mF \otimes W_q(F)$, we have $e^{m+1}(\phi)=0$ iff $\phi \in I^{m+1}F \otimes W_q(F)$.


\begin{prop}
The diagram 
$$\xymatrix{
I^n K \otimes W_q(K)\ar[r]^{\tr_*}\ar[d]^{e^{n+1}} & I^n F \otimes W_q(F)\ar[d]^{e^{n+1}}\\
H_2^{n+1}(K)\ar[r]^{\tr_*} & H_2^{n+1}(F)}$$
is commutative for any $n \geq 0$.
\label{Diagram}
\end{prop}

\begin{proof} Let $\phi\in I^n K \otimes W_q(K)$. By Proposition \ref{Generators} it suffices to verify the diagram commutativity for $\phi = \langle\langle a_1,\dots,a_n,b]]$ where $a_1,\dots,a_n \in F^\times$ and $b \in K$. By the Frobenius reciprocity \cite[Proposition 20.2]{EKM}, we have$$\tr_*(\phi)=\langle\langle a_1,\dots,a_n\rangle\rangle \otimes \tr_*([1,b]).$$Since $\tr_*([1,b])\equiv [1,\tr(b)] \pmod{IF \otimes W_q(F)}$ \cite[Lemma 34.14]{EKM}, it follows that$$\tr_*(\phi)\equiv \langle\langle a_1, \cdots, a_n, \tr(b)]] \pmod{I^{n+1}F \otimes W_q(F)},$$and thus 

\begin{eqnarray*}
e^{n+1}(\tr_*(\phi)) & = & e^{n+1}(\langle\langle a_1, \cdots, a_n, \tr(b)]]) \\&&\\ & = & \overline{\tr(b)\frac{da_1}{a_1}\wedge \cdots \wedge\frac{da_n}{a_n}} \\&&\\ & = & \tr_*\left( \overline{b\frac{da_1}{a_1}\wedge \cdots \wedge\frac{da_n}{a_n}}\right) \\&&\\ & = & \tr_*\circ e^{n+1} (\phi).
\end{eqnarray*}
\end{proof}

Since $\lambda \wedge \wp(u) \equiv \wp(\lambda\wedge u) \pmod{d\Omega_F^{m+k-1}}$, and $\lambda \wedge dv=d(\lambda\wedge v)$ for any $\lambda \in \nu_F(m)$ and arbitrary $u \in \Omega_F^k$ and $v\in \Omega_F^l$, it follows that the exterior product induces an action of $\nu_F(m)$ on the groups $H_2^{n+1}(F)$ in a natural way:
\begin{eqnarray*}
\wedge : \nu_F(m) \times H_2^{n+1}(F) & \rightarrow & H_2^{m+n+1}(F)\\
(\lambda,\overline{w}) & \mapsto & \overline{\lambda \wedge w}.
\end{eqnarray*}

Let $B$ be a biquaternion $K$-algebra such that ${\rm cor}_{K/F}(B)=0$. This condition on corestriction means that the algebra $B$ is defined over $F$ up to Brauer equivalence. Let $\varphi$ be an Albert $K$-quadratic form such that $C(\varphi)\cong M_2(B)$ (recall that an Albert form is a non singular quadratic form of dimension $6$ and trivial Arf invariant). This form $\varphi$ is unique up to scalar multiplication \cite{MammoneShapiro:1989}. Note that $e^2(\phi)=[B]$, where $[B]$ is the class of $B$ in $H^2_2(F)$.

The following proposition is the characteristic $2$ analogue of  \cite[Lemma 4.1]{Barry}:

\begin{prop}
\begin{enumerate}
\item[(1)] $\tr_*(\varphi) \in I^2 F \otimes W_q(F)$.
\item[(2)] $\forall {\lambda \in K^\times},\; e^3(\tr_*(\varphi))=e^3(\tr_*(\lambda \varphi))+\tr_*(\dlog \lambda \wedge [B])$, where $\dlog\lambda$ denotes $\frac{d \lambda}{\lambda}$ in $\nu_K(1)$ for any $\lambda \in K^{\times}$.

\item[(3)] For any $\lambda \in K^\times$, $\tr_*(\lambda \phi)=0$ if and only if $e^3(\tr_*(\varphi))=\tr_*(\dlog \lambda \wedge [B])$.
\end{enumerate}
\label{properties}
\end{prop}

\begin{proof}
For part (1), we follow the same proof as in characteristic not $2$ using Proposition \ref{Diagram} for $n=1$. 

%
%

For Part (2), since $\tr_*(\varphi)=\tr_*(\lambda \varphi)+\tr_*(\langle 1,\lambda \rangle \otimes \varphi)$, it follows that $e^3(\tr_*(\varphi))=e^3(\tr_*(\lambda \varphi))+e^3(\tr_*(\langle 1,\lambda \rangle \otimes \varphi))$.
Since the diagram 
$$\xymatrix{
I^2 K \otimes W_q(K)\ar[r]^{\tr_*}\ar[d]^{e^3} & I^2 F \otimes W_q(F)\ar[d]^{e^3}\\
H_2^{3}(K)\ar[r]^{\tr_*} & H_2^{3}(F)}$$
is commutative, we get 
$e^3(\tr_*(\langle 1,\lambda \rangle \otimes \varphi))=\tr_*(e^3(\langle 1,\lambda \rangle \otimes \varphi))=\tr_*(\dlog \lambda \wedge e^2(\varphi))=\tr_*(\dlog \lambda \wedge [B])$.

For Part (3), let $\lambda \in K^{\times}$. We have $e^3(\tr_*(\varphi))=\tr_*(\dlog \lambda \wedge [B])$ iff $e^3(\tr_*(\lambda\varphi))=0$ iff $\tr_*(\lambda\phi)\in I^3F \otimes W_q(F)$. This is equivalent to $\tr_*(\lambda\phi)=0$ by the Hauptsatz of Arason-Pfister for non singular quadratic forms \cite[Proposition 6.4]{Laghribi:2011}.
\end{proof}

Now statement (2) of Proposition \ref{properties} allows us to attach to $B$ an invariant as follows:
\begin{defn}
The invariant $\delta_{K/F}(B)$ is the class of $e^3(\tr_*(\varphi))$ in the group 
\[
H_2^3(F)/\tr_*(\dlog K^* \wedge [B]).
\]
\end{defn}

The descent criterion of $B$ to $F$ is as follows:

\begin{thm}
The $K$-algebra $B$ has a descent to $F$ if and only if $\delta_{K/F}(B)=0$.
\end{thm}

\begin{proof} We proceed as in \cite{Barry}. By Proposition \ref{properties}, $\delta_{K/F}(B)=0$ if and only if there exists $\lambda \in K^{\times}$ such that $\tr_*(\lambda \phi)=0$. 

Suppose that $B$ has a descent to $F$, this means that there exists a biquaternion $F$-algebra $B_0$ such that $B\cong B_0\otimes K$. Let $\phi_0$ be an Albert quadratic form over $F$ such that $C(\phi_0)=M_2(B_0)$. Then, $\phi$ and $(\phi_0)_K$ have the same Clifford invariant. It follows from \cite{MammoneShapiro:1989} that $\phi\cong \lambda (\phi_0)_K$ for a suitable $\lambda \in K^{\times}$. Hence, $\tr_*(\lambda \phi)=0$.

Conversely, suppose that $\tr_*(\lambda \phi)=0$ for some $\lambda\in K^{\times}$. By Proposition \ref{exaseq}, in the case $n=1$, there exists $\phi_0 \in I F \otimes W_q(F)$ such that $\lambda \phi$ is Witt equivalent to $(\phi_0)_K$. Since the extension $K/F$ is excellent, we may suppose that $\lambda \phi \cong (\phi_0)_K$. Let us write $\phi_0=a_1[1,b_1]\perp a_2[1,b_2]\perp a_3[1, b_3]$. The Arf invariant of $(\phi_0)_K$ is trivial, hence $b_1+b_2+b_3+\epsilon a \in \wp(F)$ for $\epsilon \in \{0, 1\}$ (recall that $K=F[\alpha]$ where $\wp(\alpha)=a \in F \setminus \wp(F)$). We have $(\phi_0)_K\cong (a_1[1,b_1]\perp a_2[1,b_2]\perp a_3[1, b_3+\epsilon a])_K$, and thus we may suppose that $\phi_0$ is an Albert form. Using the Clifford algebra, we get that $B$ is Brauer equivalent to $(B_0)_K$, where $B_0$ is the biquaternion $F$-algebra satisfying $C(\phi_0)=M_2(B_0)$. By dimension count we deduce that $B \cong B_0\otimes K$.
\end{proof}

\begin{rem}
Let $A$ be a central simple algebra of degree $8$ and exponent $2$ over $F$. Suppose that $K$ is contained in $A$. Denote by $B = C_A K$ the centralizer of $K$ in $A$. The algebra $A$  admits a decomposition of the form $A\cong [a, a')\otimes A'$, for some $a' \in F$ and some subalgebra $A' \subset A$ over $F$, if and only if $B$ has a descent to $F$. That is, $A$ admits such a decomposition  if and only if $\delta_{K/F}(B) = 0$. Hence, $\delta_{K/F}(B)\ne 0$ if $A$ is indecomposable.
\end{rem}
In the following example we show that there exists a decomposable  algebra of degree $8$ having $K$ as subfield such that the centralizer of $K$, a biquaternion algebra over $K$, has no descent to $F$.
\begin{exmpl}
Let $D$ be an indecomposable algebra of degree $8$ and exponent $2$ over $F$. Suppose that $K$ is contained in $D$. It is well-known that $M_2(D)$ is a tensor product of four quaternion algebras. The proof of this fact given in  \cite[Thm. 5.6.38]{Jacobson1996} shows that one may find quaternion algebras $Q_1, Q_2, Q_3, Q_4$ over $F$ such that $M_2(D) \cong Q_1 \otimes Q_2 \otimes Q_3 \otimes Q_4$ with  $K\subset Q_4$. Setting $A=Q_1 \otimes Q_2 \otimes Q_3$, we may check that $A$ is a division algebra.  Since $D_K$ is of index $4$, the index of $A_K$ is also $4$. So $A$ contains a subfield isomorphic to $K$. A result of Merkurjev shows that there is  no quaternion subalgebra of $A$ containing $K$, see \cite[Cor. 4.5]{Barry16} for char$(F)\ne 2$. The characteristic $2$ case works exactly by the same arguments. Therefore, if $B$ is the centralizer of $K$ in $A$, the invariant $\delta_{K/F}(B)$ is not trivial.
\end{exmpl}

\section{Indecomposable algebras in cohomological dimension $3$}
Recall that the cohomological 2-dimension $\operatorname{cd}_2(\ell)$ of a field $\ell$ is by definition the smallest integer such that for every $n> \operatorname{cd}_2(\ell)$ and every finite field extension $L/\ell$ we have $H_2^n(L)=0$. This latter equality holds if and only if $I^{n-1}L \otimes W_qL=0$ (see \cite[Fact 16. 2]{EKM}). 
Note that in \cite[Section 101]{EKM} the group is denoted by $H^{n,n-1}(L,\mathbb{Z}/2\mathbb{Z})$, but in the case of fields of characteristic $2$, $H^{n,n-1}(L,\mathbb{Z}/2\mathbb{Z})=H_2^n(L)$. See also \cite{Izhboldin:2000} and \cite[P. 152]{GaribaldiMerkurjevSerre}.

In this section, we  construct an example  of an indecomposable algebra of exponent $2$ and degree $8$  over a field of  $2$-cohomological dimension $3$. Notice that  every central simple algebra of exponent $2$ over $F$ decomposes into tensor product of quaternion algebras if $\operatorname{cd}_2(F)= 2$ (see \cite{Kahn} if char$(F)\ne 2$ and \cite{BarChap15} if char$(F)=2$). Hence, as it is the case in characteristic different from $2$ \cite{Barry},  this example shows that $3$ is the lower bound  of  the $2$-cohomological dimension for the existence of  indecomposable algebras of exponent $2$. 

Let $A$ be a central simple algebra over $F$. Recall that $\operatorname{TCH}^2(\operatorname{SB}(A))$ denotes the torsion in the Chow group of cycles of codimension $2$ over the Severi-Brauer variety $\operatorname{SB}(A)$  modulo rational equivalences.  If $A$ is of prime exponent $p$ and index $p^n$ (except the case $p = 2 = n$) Karpenko shows in \cite[Proposition 5.3]{Kar98} that $A$ is indecomposable if $\operatorname{TCH}^2(\operatorname{SB}(A))\ne 0$. Examples of such indecomposable algebras, independently to the characteristic of the base field, are given in \cite[Corollary 5.4]{Kar98}. The following result appeared in  \cite[Theorem 1.3]{Barry} under the assumption char$(F)\ne 2$:
\begin{thm} \label{DimCoh}
Let $A$ be a central simple algebra of degree $8$ and exponent $2$ over $F$ such that $\operatorname{TCH}^2(\operatorname{SB}(A)) \ne 0$. There exists an extension $\mathsf M$ of $F$ with $\operatorname{cd}_2(\mathsf M)=3$ such that $A_{\mathsf M}$ is indecomposable.
\end{thm}     
For the proof we first prove:
\begin{lem}\label{lemDimCoh}
Let $F$ be a field of char$(F)=2$ such that $\operatorname{cd}_2(F)> 2$. Then there exists an extension $\mathsf M$ of $F$ such that $\operatorname{cd}_2(\mathsf M)=3$, $\mathsf M$ is 2-special, and the restriction of any anisotropic quadratic 3-fold Pfister form over $F$ to $\mathsf M$ is anisotropic.
\end{lem}
\begin{proof}
Following a construction due to Merkurjev, we define inductively  a tower of fields
\[
F=F_0\subset F_1\subset\cdots \subset F_{\infty}=\bigcup_iF_i =:\mathsf M
\]
where the field $F_{2i+1}$ is the subfield of the separable closure of $F_{2i}$ consisting of elements invariant by a fixed Sylow $2$-subgroup of the absolute Galois group of $F_{2i}$. By \cite[Proposition 101.16]{EKM} $F_{2i+1}$ is 2-special, i.e., every finite field extension of $F_{2i+1}$ is of degree a power of 2. The field $F_{2i+2}$ is the composite of all the function fields  $F_{2i+1}(\psi)$ for $\psi$ ranging over all $4$-fold Pfister forms over $F_{2i+1}$. Since $I^3 \mathsf M \otimes W_q(M)=0$ by construction,  we have $H_2^4(\mathsf M)=0$. Furthermore, the field $\mathsf M$ is $2$-special, so it follows by \cite[Example 101.17]{EKM} that $\operatorname{cd}_2(\mathsf M)<4$. Anisotropic quadratic 3-fold Pfister forms over $F$ are split by neither odd degree extensions (Springer's theorem, \cite[Corollary 18.5]{EKM}) nor function fields of 4-fold Pfister forms (The separation theorem, \cite[Theorem 26.5]{EKM}), and so their restrictions to $M$ are anisotropic. Therefore, $\operatorname{cd}_2(\mathsf M)=3$.
\end{proof}
\begin{proof}[Proof of Theorem \ref{DimCoh}]
Since the theorem was already proven for char$(F)\ne 2$, we may assume char$(F) = 2$. Notice that $\operatorname{cd}_2(F)> 2$ since $A$ is indecomposable.  Let $\mathbb F$ be an odd degree extension of $F$. It is shown in \cite[ Corollary 1.2 and Proposition 1.3]{Kar98} that the scalar extension map
\[
\operatorname{TCH}^2(\operatorname{SB}(A)) \longrightarrow \operatorname{TCH}^2(\operatorname{SB}(A)_{\mathbb F})
\]
is an injection. On the other hand, let $\psi$ be a $4$-fold Pfister quadratic form over $\mathbb F$ and let $X_\psi$ be the quadric defined by $\psi$.  We show in Theorem \ref{thtor} that $\operatorname{CH}^2(X_\psi)$ is torsion free as it is the case in characteristic different from 2 by Karpenko \cite[Theorem 6. 1]{Kar91}. Hence,  Merkurjev's Chow group computations in \cite[Theorem 6.7]{Barry} show that the scalar extension map
\[
\operatorname{TCH}^2(\operatorname{SB}(A)_{\mathbb F})  \longrightarrow \operatorname{TCH}^2(\operatorname{SB}(A)_{\mathbb F(X_\psi)}) 
\]    
is injective. 

Now, let $\mathsf M/F$ be an extension with $\operatorname{cd}_2(\mathsf M)\le3$ constructed  as in Lemma \ref{lemDimCoh}. The two latter injections show that   $\operatorname{TCH}^2(\operatorname{SB}(A)_{\mathsf M}) \ne 0$. Consequently, the algebra $A_{\mathsf M}$ is indecomposable and $\operatorname{cd}_2(\mathsf M)=3$. 
\end{proof}
This theorem allows to construct an example of a biquaternion algebra with nontrivial invariant over a field of $2$-cohomological dimension $3$:
By Theorem \ref{DimCoh}, every indecomposable algebra of degree $8$ and exponent $2$ can be scalar extended to an indecomposable algebra over a field of cohomological $2$-dimension $3$. Since indecomposable algebras of degree $8$ and exponent $2$ exist, they exist also over fields $\mathsf M$ with $\operatorname{cd}_2(\mathsf M)=3$. Let $A$ be such an indecomposable algebra over $\mathsf M$.  Let $\mathsf K\subset A$ be a separable quadratic extension of $\mathsf M$, and $B=C_A{\mathsf K}$ the centralizer of $\mathsf K$ in $A$.  Then the  invariant $\delta_{\mathsf K/\mathsf M}(B)$ is not trivial. 

\section{The behavior under odd degree field extensions}

Let $\F$ be an extension of $F$ of odd degree and $\K=\F(\alpha)$ (recall that $K=F[\alpha]$, where $\wp(\alpha)=a \in F \setminus \wp(F)$). Let $B$ be a biquaternion $K$-algebra such that ${\rm cor}_{K/F}(B)=0$. As before let $\varphi$ be an Albert quadratic form over $K$ such that $C(\varphi)\cong M_2(B)$. Our aim is to prove the following proposition whose analogue in characteristic not $2$ is a consequence of \cite[Proposition 4.7]{Barry}.

\begin{prop} If $\delta_{\K/\F}(B_{\K})=0$, then $\delta_{K/F}(B)=0$.
\label{injoddext}
\end{prop}

Our proof will be based on a lifting argument from characteristic $2$ to characteristic $0$, and then apply the analogue of Proposition \ref{injoddext} in characteristic not $2$. 

It is well-known that we can consider $F$ as the residue field of a Henselian discrete valuation ring $A$ of characteristic $0$ with maximal ideal $2A$. Let $E$ and $A^{\times}$ denote the field of fractions and the group of units of $A$, respectively. 

We recall two results due to Wadsworth \cite{Wadsworth} which will play an important role in our proof. To this end, we consider the set$$S=\{(-1)^ka^2+4b \mid k\in \Z, a\in A^{\times}, b\in A\}$$which is clearly a multiplicative subgroup of $A^{\times}$. 

\begin{lem}[{\cite[Lemma 1.6]{Wadsworth}}] There is a surjective homomorphism $\gamma: S\longrightarrow F/\wp(F)$ given by:$$(-1)^ka^2+ 4b\mapsto \overline{b/a^2} \mod{\wp(F)}.$$Furthermore, ${\rm Ker}\,\gamma=\pm A^{\times\, 2}.$
\label{lemlif}
\end{lem}

\begin{prop}[{\cite[Proposition 1.14]{Wadsworth}}] Let $V$ be a free $A$-module, $Q$ a nondegenerate quadratic form on $V$ over $A$, and $\overline{Q}$ the induced form on the $F$-vector space $V/2V$. Then, $\det Q \in S/A^{\times\,2}$ and $\Delta(\overline{Q})=\gamma(\det Q)$, where $\Delta(\overline{Q})$ denotes the Arf invariant of $\overline{Q}$.
\label{proplif}
\end{prop}

As a corollary, we get:

\begin{cor}
For any quadratic form $\psi$ over $F$ whose Witt class is in $I^nF\otimes W_q(F)$ for some $n\geq 1$, there exists a quadratic form $\psi'$ over $A$ whose Witt class is in $I^{n+1}A$ such that $\overline{\psi'}$ is Witt equivalent to $\psi$. Moreover, for $n=1$, we may choose $\psi'$ such that $\overline{\psi'}$ is isometric to $\psi$
\label{corlif}
\end{cor}

\begin{proof} We start with the proof in the case $n=1$. Suppose that $\psi \in IF\otimes W_q(F)$ and set $\dim \psi=2m$. Let $\psi'$ be a nondegenerate quadratic form over $A$ such that $\overline{\psi'} \cong \psi$. As in the proof of \cite[Proposition 1.14]{Wadsworth}, the form $\psi'$ has a symplectic basis for which $\det \psi'=(-1)^m+4b$ for some $b\in A$. Moreover, we have by Proposition \ref{proplif} that
\begin{equation}
\gamma(\det \psi')=\Delta(\psi).
\label{eqlif}
\end{equation}

The condition $\psi \in IF\otimes W_q(F)$ is equivalent to saying that $\Delta(\psi)=0$. It follows from (\ref{eqlif}) that $\det \psi' \in {\rm Ker }\,\gamma$. By Lemma \ref{lemlif} there exists $u\in A^{\times}$ such that $\det \psi'=\pm u^2$. We claim that we have precisely $\det \psi'=(-1)^m u^2$, which means that $\psi'\in I^2A$ because $\dim \psi'=2m$. Suppose that $\det \psi'=(-1)^{m+1} u^2$. Then, $(-1)^m+4b= -(-1)^mu^2$, and thus $(u+1)^2=2(u-(-1)^m2b)\in 2A$. Hence, $u+1\in 2A$. Let $c\in A$ be such that $u+1=2c$. Then, $2c^2=u-(-1)^m2b$ and thus $u\in 2A$, a contradiction. 

Now suppose that $n\geq 2$. Since $\psi$ belongs to $I^nF\otimes W_q(F)$, it is Witt equivalent to a sum $\sum_{i=1}^k\left<\left<a_{i,1}, \cdots, a_{i,n}, b_{i}\right]\right]$ for $a_{i,1}, \cdots, a_{i,n} \in F\setminus \{0\}$ and $b_{i}\in F$. For each $1\leq i\leq k$ and $1\leq j\leq n$, we choose $a'_{i,j} \in A^{\times}$ and $b'_{i}\in A$ such that $a_{i,j}=a'_{i,j} +2A$  and $b_i=b'_i+2A$. Since the quadratic form $[1, b'_i]$ is nondegenerate over $A$ and $a'_{i,j} \in A^{\times}$, it follows that $\psi':=\sum_{i=1}^k\left<\left<a'_{i,1}, \cdots, a'_{i,n}, b'_{i}\right]\right]$ is a nondegenerate quadratic form over $A$. Moreover, $\psi' \in I^{n+1}A$ and its reduction modulo $2$ is Witt equivalent to $\psi$. 
\end{proof}

\begin{sloppypar}

\noindent{\it Proof of Proposition \ref{injoddext}.} As before let $A$ be a Henselian discrete valuation ring of characteristic $0$ with maximal ideal $2A$ and residue field $F$. Let $E$ be the field of fractions of $A$. 

From Proposition \ref{properties} (3) we have the following equivalence:
\begin{equation}
\delta_{K/F}(B)=0\; \text{if and only if there exists}\; \lambda \in K^{\times}\; \text{such that}\; \tr_*(\lambda \phi)=0.
\label{rappel}
\end{equation}

(a) Recall that $K=F(\alpha)$ where $\alpha^2+ \alpha =a\in F\setminus \wp(F)$. Let $a'\in A$ be such that $a'+ 2A =a$. The polynomial $p(t)=t^2+t+a' \in A[t]$ is irreducible over $A$, and thus irreducible over $E$. The ring $A'=A[t]/(p(t))$ is local of maximal ideal $2A'$. Since $A'$ is integral over $A$, it is a Henselian, discrete valuation ring. The residue field of $A'$ is isomorphic to $K$, and the quotient field is isomorphic to $E':=E[t]/(p(t))$ (this argument was used in \cite[middle of page 1339]{Baeza2}). We write $A'=A[\epsilon]$ for $\epsilon$ a root of $p(t)$.

(b) The extension $\F/F$ is separable since $[\F:F]$ is odd. Hence, $\F=F(\overline{\theta})$ for a suitable $\overline{\theta}$. Let $\overline{q}(t)=t^n+\overline{a}_{n-1}t^{n-1}+\cdots + \overline{a}_1 t+ \overline{a}_0\in F[t]$ be the minimal polynomial of $\overline{\theta}$ over $F$. The polynomial $q(t)= t^n+ a_{n-1}t^{n-1}+\cdots + a_1 t+ a_0\in A[t]$ is irreducible over $E$. As in (a), $A''=A[t]/(q(t))$ is a Henselian discrete valuation ring of maximal ideal $2A''$, residue field isomorphic to $\F$ and quotient field isomorphic to $E''=E[t]/(q(t))$. We write $A''=A[\epsilon'']$ for $\epsilon''$ a root of $q(t)$.

We summarize the points (a) and (b) in the following diagram:

$$\begin{tabular}{ccccc}
$E''={\rm Frac}(A'')$ & $\longleftarrow$ &  $A''=A[\epsilon'']$ & $\longrightarrow $ & $\F$\\
$\uparrow$ &  & $\uparrow$  & &  $\uparrow$\\ $E={\rm Frac}(A)$  & $\longleftarrow $ & $A$ & $\longrightarrow $ & $F=A/2A$\\ 
$\downarrow$ &  & $\downarrow$  & &  $\downarrow$\\ $E'={\rm Frac}(A')$ & $\longleftarrow$ & $A'=A[\epsilon']$ & $\longrightarrow$ & $K=F(\alpha)$
\end{tabular}$$

(c) As in cases (a) and (b), since $\K/\F$ is quadratic, the field $\K$ is the residue field of a Henselian, discrete valuation ring isomorphic to $A''[\epsilon']$. Its quotient field is isomorphic to $E'.E''$.

(d) Let $\phi'$ be a $A'$-quadratic form of dimension $6$ whose reduction modulo $2$ is isometric to $\phi$. By Corollary \ref{corlif} (applied to the ring $A'$), $\phi'$ is an Albert form over $A'$. Since ${\rm cor}_{K/F}(B)=0$, there exists a quadratic form $\overline{\psi}\in IF \otimes W_q(F)$ such that $\phi \perp \overline{\psi}_K \sim \overline{\gamma}$ for $\overline{\gamma}\in I^2K \otimes W_q(K)$. By Corollary \ref{corlif}, there exist $\psi \in I^2A$ and $\gamma \in I^3A'$ whose reductions modulo $2$ are Witt equivalent to $\overline{\psi}$ and $\overline{\gamma}$, respectively. Since $A'$ is Henselian, it follows from \cite[Satz 3.3]{Knebusch} that $\phi' \perp \psi_{A'} \sim \gamma$. In particular
\begin{equation}
\phi' \perp \psi_{E'} \in I^3E'.
\label{eqlift}
\end{equation}

Let $B'$ be the biquaternion $E'$-algebra satisfying $M_2(B')=C(\phi')$. It follows from (\ref{eqlift}) that ${\rm cor}_{E'/E}(B')=0$.

(e) Now suppose that $\delta_{\K/\F}(B_{\K})=0$. We have to prove that $\delta_{K/F}(B)=0$. 

By (\ref{rappel}) there exists $\lambda\in \K^{\times}$ such that $\tr_*(\lambda \phi)=0$, where $\tr_*$ is the transfer map induced by the trace map with respect to the extension $\K/\F$. It follows from Proposition \ref{exaseq} in the case $n=1$ that there exists $\overline{\psi}\in I\F \otimes W_q(\F)$ such that $\lambda\phi \sim \overline{\psi}_{\K}$. Let $\psi$ be a lifting of $\overline{\psi}$ to $A''$ and $\lambda' \in A''[\epsilon']$ such that $\overline{\lambda'}=\lambda$. Since $\overline{\lambda'}\,\overline{\phi'} \sim \overline{\psi}_{\K}$ and $A''[\epsilon']$ is Henselian, it follows from \cite[Satz 3.3]{Knebusch} that $\lambda' \phi' \sim \psi$ over $A''[\epsilon']$. In particular, $\lambda' \phi' \sim \psi$ over $E'.E''$. Since $\psi$ is defined over $E''$, we get $s_*(\lambda' \phi')=0$, where $s_*$ is the Scharlau transfer with respect to the quadratic extension $E'.E''/E''$, mapping 1 to 0 and $\epsilon$ to 1. This means for the algebra $B'$ that $\delta_{E'.E''/E''}(B')=0$. Since $E''/E$ is of odd degree, it follows from \cite[Proposition 4.7]{Barry} that $\delta_{E'/E}(B')=0$.

By \cite[Lemma 4.1(3)]{Barry} there exists $\mu \in E'$ such that
\begin{equation}
s_*(\mu \phi')=0
\label{reduction}
\end{equation}
where $s_*$ is the Scharlau transfer with respect to the quadratic extension $E'/E$, mapping $\epsilon$ to $1$ and $1$ to $0$. Without loss of generality, we may suppose that $\mu \in A'$. Moreover, using the Frobenius reciprocity, we may reduce to the case where $\mu$ is a unit. Hence, $\mu\phi'$ is a regular quadratic form over $A'$.  

Let $s':A'\rightarrow A$ be the $A$-linear map satisfying $s'(1)=0$ and $s'(\epsilon')=1$. The relation (\ref{reduction}) is equivalent to saying that $(s'_*(\mu \phi'))_{E}=0$. Hence, $s'_*(\mu \phi')=0$ by \cite[Corollary 3.3]{MilnorHusemoller}. It follows from \cite[Theorem 5.2, Chapter 5]{Baeza} that $\mu \phi' \sim \phi''$ for a suitable $\phi'' \in W_q(A)$. Reducing modulo $2$, we get $\overline{\mu} \phi \sim \overline{\phi''}$, where $\overline{\phi''}$ is a quadratic form over $F$. Hence, $\tr_*(\overline{\mu} \phi)=0$, which implies by (\ref{rappel}) that $\delta_{K/F}(B)=0$.\qed

\end{sloppypar}

\section{An injectivity result} 

Let $A$ be a central simple algebra over $F$ of degree $8$ and exponent $2$ containing $K=F[\alpha]$. Let $B$ be the centralizer of $K$ in $A$. Let $t$ be an indeterminate over $F$ and $A'$ the division $F(t)$-algebra Brauer equivalent to $A\otimes_F [a, t)$. Clearly, $A'\otimes_{F(t)}K(t)$ is Brauer equivalent to $B\otimes_K K(t)$. 

Our aim in this section is to prove the following proposition that extends \cite[Proposition 4.9]{Barry} to characteristic $2$. The motivation for this proposition is the construction of central simple algebras of degree 8 and exponent 2 of nonzero $\Delta$-invariant. This invariant is currently only defined for fields of characteristic not 2 (see \cite{GaribaldiParimalaTignol}), but we believe its definition will be extended to the characteristic 2 case as well, and this proposition will come in handy when it is done. In addition, it demonstrates the use of Arason residue map, which is of independent interest.

\begin{prop} The restriction map$$H^3_2(F)/\tr_*(\dlog K^{\times} \wedge [B]) \longrightarrow H^3_2(F(t))/\dlog F(t)^{\times} \wedge [A']$$is well-defined and injective.
\label{Injection}
\end{prop}

The proof of this proposition uses the existence of the residue map from the group $H^3_2(F(t))'$ to $H_2^2(F)$, where $H^3_2(F(t))'=\nu_F(2)\wedge H_2^1({\mathcal O})$ and ${\mathcal O}$ is the ring of the $t$-adic valuation of $F(t)$. 

For the rest of this section, let ${\mathcal O}$ be a discrete valuation ring of characteristic $2$ with quotient field $L$ and residue field $\overline{L}$. Let ${\mathcal O}^{\times}$ be the group of units of ${\mathcal O}$ and $\pi$ a uniformizer. Let $W_q(L)'$ be the $W(L)$-submodule of $W_q(L)$ generated by the forms $[1, a]$ where $a\in {\mathcal O}$.

The following proposition is inspired by a result of Arason \cite[Satz 8]{Arason:1979} and gives a presentation by generators and relations of $W_q(L)'$.

\begin{prop} There exists a surjective homomorphism$$S:W(L)\otimes_{\Z/2\Z}{\mathcal O}/\wp({\mathcal O}) \longrightarrow W_q(L)'$$of $W(L)$-modules given by: $\left< \alpha \right> \otimes [d]\mapsto \alpha [1,d]_L$, and whose kernel is the $W(L)$-submodule generated by the elements $\langle 1, q\rangle\otimes [d]$ such that $q\in L^{\times}$, $d\in {\mathcal O}$ and $q\in D_L([1,d])$.
\label{relatgenera}
\end{prop}

\begin{proof} First of all the map $S$ is well-defined since $[1,a]_L$ is hyperbolic for $a\in \wp({\mathcal O})$, and the forms $[1, a_1+a_2]_L$ and $[1,a_1]_L\perp [1,a_2]_L$ are Witt equivalent for any $a_1, a_2\in {\mathcal O}$. 

Let ${\mathcal N}$ be the $W(L)$-submodule of $W(L)\otimes_{\Z/2\Z}{\mathcal O}/\wp({\mathcal O})$ generated by the elements $\langle 1,q\rangle \otimes[d]$ such that $d\in {\mathcal O}$ and $q\in D_L([1,d])$. 

Since the form $\langle 1,q\rangle \otimes[1,d]_L$ is hyperbolic when $q\in D_L([1,d])$, it follows that ${\mathcal N}\subset {\rm Ker}\,S$. It remains to prove the opposite inclusion ${\rm Ker}\,S \subset {\mathcal N}$. 

Let $A=\sum_{i=1}^n\left<\alpha_i\right>\otimes [d_i]\in {\rm Ker}\,S$ for $\alpha_i\in L^{\times}$ and $d_i\in {\mathcal O}$. We proceed by induction on $n$ to prove that $A\in {\mathcal N}$. 

If $n=1$, then $A\in {\rm Ker}\,S$ implies that $\alpha_1[1, d_1]_L$ is hyperbolic. Then, it is easy to prove that $d_1\in \wp({\mathcal O})$, and thus $A=0$.

Suppose that $n>1$ and any element of ${\rm Ker}\,S$ which is a sum of less that $n$ elements $\left<\alpha\right>\otimes [d]$, with $\alpha \in L^{\times}$ and $d\in {\mathcal O}$, belongs to ${\mathcal N}$. 

Without loss of generality, we may suppose that $\alpha_1, \cdots, \alpha_n \in {\mathcal O}\setminus \{0\}$. Since $A\in {\rm Ker}\,S$, the quadratic form $Q:=\alpha_1[1, d_1]\perp \cdots \perp \alpha_n[1, d_n]$ is hyperbolic over $L$, in particular it is isotropic. 

Since $L$ is the field of fractions of ${\mathcal O}$, there exists $x_i\in D_{\mathcal O}(\alpha_i[1,d_i])$ for any $1\leq i\leq n$ such that 
$\sum_{i=1}^n x_i=0$. Without loss of generality, we may suppose that $x_i\neq 0$ for any $1\leq i\leq n$.\vskip1.5mm

\noindent{\it (1) Suppose that at least one of $x_1, \cdots, x_n$ is a unit.} We may suppose that $x_1$ is a unit. Using the standard isometries
\begin{eqnarray*}
[a,b]\perp [c,d]\cong [a+c, b]\perp [c, b+d] & \text{and}& x[a,b] \cong [xa, x^{-1}b]
\end{eqnarray*}
for any $a, b, c, d\in L$ and $x\in L\setminus\{0\}$, we deduce

\begin{eqnarray}\nonumber Q_L & \cong & [x_1, x_1^{-1}d_1] \perp [x_2, x_2^{-1}d_2]\perp \cdots \perp [x_n, x_n^{-1}d_n]\\\label{eqiso1} & \cong & [0, x_1^{-1}d_1] \perp [x_2, x_2^{-1}d_2+x_1^{-1}d_1]\perp \cdots \perp [x_n, x_n^{-1}d_n+ x_1^{-1}d_1]\\\nonumber & \cong & [0,0] \perp 
x_2[1, d_2+ x_2x_1^{-1}d_1]\perp \cdots \perp x_n[1, d_n+ x_n x_1^{-1}d_1].
\end{eqnarray}
Since $Q_L$ is hyperbolic, it follows from (\ref{eqiso1}) that the quadratic form
\[x_2[1, d_2+ x_2x_1^{-1}d_1]\perp \cdots \perp x_n[1, d_n+ x_n x_1^{-1}d_1]\] is hyperbolic over $L$, which means that $$B:=\sum_{i=2}^n\left<x_i\right>\otimes [d_i+ x_i x_1^{-1}d_1]\in {\rm Ker}\,S.$$

We deduce by the induction hypothesis that $B\in {\mathcal N}$. Moreover, the hyperbolicity of $\alpha_i[1, d_i]_L\perp x_i[1, d_i]_L$ implies that
\begin{equation}
\left<\alpha_i\right>\otimes [d_i] \equiv \left<x_i\right>\otimes [d_i] \mod{\mathcal N}.
\label{revequa1}
\end{equation}

We also have
\begin{eqnarray}\label{revequa2}
\left<x_i\right>\otimes [x_ix_1^{-1}d_1] & \equiv & \left<x_i(x_ix_1^{-1}d_1)\right>\otimes [x_ix_1^{-1}d_1] \mod{\mathcal N}\\\nonumber
& = & \left<x_1^{-1}d_1\right>\otimes [x_ix_1^{-1}d_1].
\end{eqnarray}

Now combining (\ref{revequa1}) and (\ref{revequa2}) with the fact that $B=\sum_{i=2}^n\left<x_i\right>\otimes [d_i]+ \sum_{i=2}^n \left<x_i\right>\otimes [x_i x_1^{-1}d_1]\in {\mathcal N}$, we deduce modulo ${\mathcal N}$ that
\begin{eqnarray*}
A & \equiv & \left<\alpha_1\right>\otimes[d_1]+ \sum_{i=2}^n \left<x_1^{-1}d_1\right>\otimes [x_ix_1^{-1}d_1]\\ & \equiv & 
\left<x_1\right>\otimes[d_1]+ \sum_{i=2}^n \left<x_1^{-1}d_1\right>\otimes [x_ix_1^{-1}d_1]\\ & \equiv & \left<x_1d_1\right>\otimes[d_1]+ \sum_{i=2}^n \left<x_1^{-1}d_1\right>\otimes [x_ix_1^{-1}d_1]\\& = & \left<x_1^{-1}d_1\right>\otimes[d_1]+ \sum_{i=2}^n \left<x_1^{-1}d_1\right>\otimes [x_ix_1^{-1}d_1]\\& =& \left<x_1^{-1}d_1\right>\otimes\left[\sum_{i=1}^n x_ix_1^{-1}d_1\right].
\end{eqnarray*}

Since ${\mathcal N} \subset {\rm Ker}\,S$ and $A\in {\rm Ker}\,S$, it follows that $C:= \left<x_1^{-1}d_1\right>\otimes\left[\sum_{i=1}^n x_ix_1^{-1}d_1\right]\in {\rm Ker}\,S$. By the case $n=1$, we conclude that $C\in {\mathcal N}$. Consequently, $A\in {\mathcal N}$.

{\it (2) Suppose that $x_1, \cdots, x_n\in \pi{\mathcal O}$.} Hence, $x_i=\pi^{\epsilon_i}y_i$ for $y_i$ a unit and $\epsilon_i>0$. Without loss of generality, we may suppose $\epsilon_1\leq \epsilon_i$ for any $2\leq i\leq n$. Since we have \[\left<\alpha_i\right>\otimes [d_i]\equiv \left<\pi^{\epsilon_i} y_i\right>\otimes [d_i]\, \mod{\mathcal N}\] for any $1\leq i\leq n$, it follows that
\begin{equation}
A\equiv \sum_{i=1}^n\left<\pi^{\epsilon_i} y_i\right>\otimes [d_i] \mod{\mathcal N}.
\label{revequa3}
\end{equation}
Obviously, the hyperbolicity of $Q_L$ implies that the form\[y_1[1, d_1]\perp \pi^{\epsilon_2-\epsilon_1}y_2[1, d_2]\cdots \perp \pi^{\epsilon_n-\epsilon_1}y_n[1, d_n]\] is also hyperbolic over $L$, which means that $A':=\sum_{i=1}^n\left< \pi^{\epsilon_i-\epsilon_1}y_i\right>\otimes [d_i]\in {\rm Ker}\,S$. Since $y_1+\sum_{i=2}^n\pi^{\epsilon_i-\epsilon_1}y_i=0$ and $y_1$ is a unit, it follows from the case (1) that $A'\in {\mathcal N}$. Hence, $\left<\pi^{\epsilon_1}\right>. A'\in {\mathcal N}$ and we conclude by (\ref{revequa3}) that $A\in {\mathcal N}$.

\end{proof}

For any $a\in L^{\times}$, we have $\langle a\rangle \cong \langle \pi^iu\rangle$ where $i=0$ or $1$ and $u\in {\mathcal O}^{\times}$. There exist group homomorphisms $\partial_{\pi}^1, \partial_{\pi}^2:W(L)\rightarrow W(\overline{L})$ defined on generators as follows:
\[\partial_{\pi}^k(\langle \pi^iu\rangle)=\begin{cases} \langle \overline{u} \rangle & \text{if}\; k\not\equiv i \pmod{2}\\ 0 & \text{if}\; k\equiv i \pmod{2}
\end{cases}\]
(see \cite[Lemma (1.1), page 85]{MilnorHusemoller}). We call $\partial_{\pi}^1$ and $\partial_{\pi}^2$ the first and the second residue homomorphisms associated with the valuation of ${\mathcal O}$. Note that $\partial_{\pi}^2$ depends on the choice of the uniformizer $\pi$.

\begin{prop} We keep the notations and hypotheses as in Proposition \ref{relatgenera}.  There exist two group homomorphisms $\delta_{\pi}, \Delta_{\pi}:W_q(L)'\rightarrow W_q(\overline{L})$ given as follows:$$\delta_{\pi}(B\otimes [1,d])=\partial_{\pi}^2(B) \otimes [1,\overline{d}]$$and$$\Delta_{\pi}(B\otimes [1,d])=(\partial_{\pi}^1+\partial_{\pi}^2)(B) \otimes [1,\overline{d}],$$where for $d\in {\mathcal O}$, we denote by $\overline{d}$ its residue class in $\overline{L}$.
\label{Residues}
\end{prop}

\begin{proof} For any $d_1, d_2\in {\mathcal O}$, if $[1,d_1]_L\cong [1,d_2]_L$, then $[1,\overline{d_1}]\cong [1,\overline{d_2}]$. Hence, we have bi-additive maps $W(L)\times {\mathcal O}/\wp({\mathcal O}) \rightarrow W_q(\overline{L})$ given by:$$(B, [d])\mapsto \partial_{\pi}^2(B) \otimes [1, \overline{d}]$$ $$(B, [d])\mapsto (\partial_{\pi}^1+\partial_{\pi}^2)(B) \otimes [1, \overline{d}].$$

This induces group homomorphisms $\lambda, \gamma: W(L)\otimes_{\Z/2\Z} {\mathcal O}/\wp({\mathcal O}) \rightarrow W_q(\overline{L})$ given by:$$\lambda(B\otimes [d])=\partial_{\pi}^2(B) \otimes [1, \overline{d}]$$ $$\gamma(B\otimes [d])=(\partial_{\pi}^1+\partial_{\pi}^2)(B) \otimes [1, \overline{d}].$$

Moreover, let $d\in {\mathcal O}$ and $q\in L^{\times}$ be such that $q\in D_L([1,d])$, and let $a \in L^{\times}$. 

(1) If $q=u\pi$ for $u\in {\mathcal O}^{\times}$, then the condition $q\in D_L([1,d])$ implies $[1, \overline{d}]=0$. Hence, $\lambda(a\langle 1, q\rangle \otimes [d])=0$ and $\gamma(a\langle 1, q\rangle \otimes [d])=0$.

(2) If $q$ is a unit, then $\overline{q}\in D_{\overline{L}}([1, \overline{d}])$ and thus $\langle 1, \overline{q}\rangle \otimes [1, \overline{d}]=0$. 

-- If $a$ is a unit, then $\lambda(a\langle 1,q\rangle \otimes [d])=0$ since $\partial_{\pi}^2(a\langle 1,q\rangle)=0$, and $\gamma(a\langle 1, q\rangle \otimes [d])=\overline{a}\langle 1, \overline{q}\rangle \otimes [1, \overline{d}]=0$.

-- If $a=\pi b$ for $b\in {\mathcal O}^{\times}$. Then,$$\lambda(a\langle 1, q\rangle \otimes [d])=\gamma(a\langle 1, q\rangle \otimes [d])=\overline{b}\langle 1, \overline{q}\rangle \otimes [1, \overline{d}]=0.$$

This proves that the maps $\lambda$ and $\gamma$ vanish on the kernel of$$S:W(L)\otimes_{\Z/2\Z}{\mathcal O}/\wp({\mathcal O}) \longrightarrow W_q(L)',$$and thus this induces, by Proposition \ref{relatgenera}, group homomorphisms $\delta_{\pi}:W_q(L)' \rightarrow W_q(\overline{L})$ and $\Delta_{\pi}:W_q(L)' \rightarrow W_q(\overline{L})$ given by: $\delta_{\pi}(B\otimes [1,d])=\partial_{\pi}^2(B) \otimes [1, \overline{d}]$ and $\Delta_{\pi}(B\otimes [1,d])=(\partial_{\pi}^1+\partial_{\pi}^2)(B) \otimes [1, \overline{d}]$.
\end{proof}

For any integer $n\geq 1$, let $P_n(L)'$ denote the set of $n$-fold quadratic Pfister forms $\langle\langle a_1, \cdots, a_{n-1},b]]$ with $a_1, \dots, a_{n-1}\in L^{\times}$ and $b\in {\mathcal O}$. Let $GP_n(L)'$ be the set $L\cdot P_n(L)'$. We take $I^n_q(L)'=I^{n-1}L \otimes \{[1,a]\mid a\in {\mathcal O}\}$ and $\overline{I}^n_q(L)'=I^n_q(L)'/I^{n+1}_q(L)'$. Clearly, $I^n_q(L)'$ is additively generated by $GP_n(L)'$. 

\begin{lem} We keep the same notations and hypotheses as in Proposition \ref{Residues}. Let $n\geq 1$ be an integer. Then:\\(1) 
$\delta_{\pi}(P_n(L)')=GP_n(\overline{L}) \cup GP_{n-1}(\overline{L})$ and $\delta_{\pi}(I^n_q(L)')=I^{n-1}_q(\overline{L})$.\\(2) $\Delta_{\pi}(P_n(L)')=GP_n(\overline{L})$ and $\Delta_{\pi}(I^n_q(L)')=I^{n}_q(\overline{L})$.\\(3) There exist well-defined group homomorphisms$$\delta: \overline{I}^{n}_q(L)' \rightarrow \overline{I}^{n-1}_q(\overline{L})$$and$$\Delta_{\pi}: \overline{I}^{n}_q(L)' \rightarrow \overline{I}^{n}_q(\overline{L}),$$ given by: $\delta(\phi + I^{n+1}_q(L)')=\delta_{\pi}(\phi)+ I^{n}_q(\overline{L})$ and $\Delta_{\pi}(\phi + I^{n+1}_q(L)')= \Delta_{\pi}(\phi)+ I^{n}_q(\overline{L})$. Moreover, the map $\delta$ is independent of the choice of the uniformizer $\pi$.
\label{filtration}
\end{lem}

\begin{proof} For (1) and (2): Using some arguments from the proof of \cite[Satz 3.1]{Arason:1975}, we get $\partial_{\pi}^2(BP_n(L)) =GBP_n(\overline{L}) \cup GBP_{n-1}(\overline{L})$ and $(\partial_{\pi}^1+\partial_{\pi}^2)(BP_n(L)) =GBP_n(\overline{L})$, where $BP_n(L)$ denotes the set of $n$-fold bilinear Pfister forms over $L$ and $GBP_n(L)=L^{\times}.BP_n(L)$. Consequently, $\partial_{\pi}^2(I^nL)=I^{n-1}\overline{L}$ and $(\partial_{\pi}^1+\partial_{\pi}^2)(I^nL)=I^{n}\overline{L}$. Now the statements (1) and (2) readily follow from the definitions of $\delta_{\pi}$ and $\Delta_{\pi}$.

For (3): The maps $\delta$ and $\Delta_{\pi}$ are well-defined by statements (1) and (2). Let $\pi'$ be another uniformizer of ${\mathcal O}$. Hence, there exists $u\in {\mathcal O}^{\times}$ such that $\pi=u\pi'$. Clearly, for any $d\in {\mathcal O}$, the form $[1, \overline{d}]$ is independent of the uniformizer. Moreover, for any form $B\in I^nL$, we have $\partial^2_{\pi}(B)+ \partial^2_{\pi'}(B)=\left<1, \overline{u}\right>\otimes \partial^2_{\pi'}(B)\in I^n\overline{L}$ because $\partial^2_{\pi'}(B) \in I^{n-1}\overline{L}$. Consequently, $\delta_{\pi}(\phi)+ I^{n}_q(\overline{L})=\delta_{\pi'}(\phi)+ I^{n}_q(\overline{L})$ for any $\phi \in I^{n}_q(L)'$, and thus the map $\delta$ is independent of the uniformizer.
\end{proof}

\begin{notation} Let $H_2^{m+1}(L)'$ denote the subgroup $\nu_L(m)\wedge H_2^1({\mathcal O})$ of $H_2^{m+1}(L)$. 
\end{notation}

Kato's isomorphism $f_{m+1}$ induces an isomorphism $g_{m+1}: H_2^{m+1}(L)' \rightarrow \overline{I}^{m+1}_q(L)'$. 

As a consequence of Proposition \ref{Residues} and Lemma \ref{filtration}, we obtain two residue maps$$\xi: H_2^{m+1}(L)' \rightarrow H_2^{m}(\overline{L})$$and$$\chi_{\pi}: H_2^{m+1}(L)' \rightarrow H_2^{m+1}(\overline{L})$$so that the following diagrams commute:

$$\begin{tabular}{cccc}
$\xymatrix{
H_2^{m+1}(L)' \ar[r]^{g_{m+1}}\ar[d]^{\xi} & \overline{I}^{m+1}_q(L)'  \ar[d]^{\delta}\\
H_2^{m}(\overline{L})\ar[r]^{f_{m}} & \overline{I}^m_q(\overline{L})}$ & & & $\xymatrix{
H_2^{m+1}(L)' \ar[r]^{g_{m+1}}\ar[d]^{\chi_{\pi}} & \overline{I}^{m+1}_q(L)'  \ar[d]^{\Delta_{\pi}}\\
H_2^{m+1}(\overline{L})\ar[r]^{f_{m+1}} & \overline{I}^{m+1}_q(\overline{L})}$
\end{tabular}$$
that is $\xi=f_{m}^{-1}\circ \delta \circ g_{m+1}$ and $\chi_{\pi}=f_{m+1}^{-1}\circ \Delta_{\pi} \circ g_{m+1}$.
\vskip2mm

Now we are able to prove Proposition \ref{Injection}.\vskip1.5mm

\noindent{\it Proof of Proposition \ref{Injection}.} Let $F(t)$ be the rational function field in the indeterminate $t$. Recall that $A'$ is the division $F(t)$-algebra Brauer equivalent to $A\otimes_F [a, t)$. We have also that $A'\otimes_{F(t)}{K(t)}$ is Brauer equivalent to $B\otimes_K {K(t)}$. Let $c=\sum_{i=1}^m\overline{a_i\frac{db_i}{b_i}}\in H_2^2(F)$ be such that $c=[A]$. Hence, we get $[A']=c+\overline{a\frac{dt}{t}}$. 

The map$$H^3_2(F)/\tr_*(\dlog K^{\times} \wedge [B]) \longrightarrow H^3_2(F(t))/\dlog F(t)^{\times} \wedge [A']$$is well-defined by the same argument as in \cite[Proof of Proposition 4.9]{Barry}. 

Now let $w \in H_2^3(F)$ be such that
\begin{equation}
w_{F(t)}= \left( c+ \overline{a\frac{dt}{t}}\right)\wedge \frac{df(t)}{f(t)}
\label{e1}
\end{equation}
for some $f(t)\in F(t)^{\times}$. Without loss of generality,  we may assume that $f(t)$ is square free. Our aim is to prove that 
$w \in \tr_*(\dlog K^{\times} \wedge [B])$. 

We consider the $t$-adic valuation on $F(t)$. Clearly, $w_{F(t)}$, $\overline{a\frac{dt}{t}\wedge \frac{df(t)}{f(t)}}$, and $c$ belong to $H_2^3(F(t))'$. By applying the map $\xi$, previously defined, to the equality (\ref{e1}), we get: 
\begin{equation}
\xi \left( c\wedge \frac{df(t)}{f(t)} \right)= \xi\left( \overline{a\frac{dt}{t}}\wedge \frac{df(t)}{f(t)}\right).
\label{e2}
\end{equation}

(i)  Suppose that $f(t)$ is a unit. On the one hand, since $\partial_t^2(\langle\langle t, f(t)\rangle\rangle)=\langle 1, f(0)\rangle$ and $\xi(c\wedge \frac{df(t)}{f(t)})=0$, it follows from (\ref{e2}) that $\overline{a\frac{df(0)}{f(0)}}=0\in H_2^2(F)$. On the other hand, we apply $\chi_t$ to the relation 
(\ref{e1}), we get $w=c\wedge \frac{df(0)}{f(0)}$. Since $\overline{a\frac{df(0)}{f(0)}}=0$, it follows that $f(0)\in D_F([1,a])=N_{K/F}(K^{\times})$. Using  \cite[Lemma 2.5]{AAB} and the Frobenius reciprocity we get $w \in \tr_*(\dlog K^{\times}\wedge [A])$. 

(ii) Suppose that $f(t)=tg(t)$. Since $\langle\langle t, f(t)\rangle\rangle\cong \langle\langle t, g(t)\rangle\rangle$, it follows that $\partial_t^2(\langle\langle t, f(t)\rangle\rangle)=\langle 1, g(0)\rangle$, and thus $\xi(\overline{a\frac{dt}{t}}\wedge \frac{df(t)}{f(t)})=\overline{a\frac{dg(0)}{g(0)}}$. Moreover, because $\partial_t^2(\langle\langle b_i, f(t)\rangle\rangle)=g(0)\langle 1, b_i\rangle$ for any  $1\leq i \leq m$, we get $\xi(c\wedge \frac{df(t)}{f(t)})=c$. Hence, we obtain by (\ref{e2}) that $c=\overline{a\frac{dg(0)}{g(0)}}$, i.e., $A$ is Brauer equivalent to $[a, g(0))$. This implies that $B\sim_{Br} A\otimes K\sim_{Br} 0$, and thus $\tr_*(\dlog K^{\times} \wedge [B])=0$.

Moreover, applying $\chi_t$ to the equation (\ref{e1}) implies that 
\[
w=\chi_t(c \wedge \frac{df(t)}{f(t)})+\chi_t(\overline{a\frac{dt}{t}}\wedge\frac{df(t)}{f(t)})=\overline{a\frac{dg(0)}{g(0)}}\wedge  \frac{dg(0)}{g(0)}=0.
\]\qed

\vspace{0.5cm}
\begin{appendix}

\section{Torsion in codimension 2 Chow groups of projective quadrics in characteristic two}

A regular quadratic form means a quadratic form which is either nonsingular or singular whose quasilinear part is anisotropic of dimension $1$. Our aim is to prove the following:

\begin{thm}
Let $\phi$ be a regular quadratic form over $F$ of dimension $>8$, and $X_{\phi}$ its projective quadric. Then, the group $\CH^2(X_{\phi})$ is torsion free.
\label{thtor}
\end{thm}

The analogue of this theorem is well-known in characteristic not $2$ and is due to Karpenko \cite{Kar91}. Our proof will proceed in three steps and follows many arguments given by Karpenko in \cite{Kar91} and \cite{Kar95}, which apply to our case since we consider smooth projective quadrics. We decide to include a proof since there is no reference concerning Theorem \ref{thtor}. 

\medskip

\noindent{\bf Step 1:} Our aim in this step is to prove the following:

\begin{prop} Let $\phi$ be a regular quadratic form over $F$ of dimension $>8$. Then, there exists a purely transcendental extension $L/F$ and a regular quadratic form $\phi'$ over $L$ of dimension $9$ such that $\TCH^2(X_{\phi})\simeq \TCH^2(X_{\phi'})$.
\label{prtor1}
\end{prop}

\begin{proof} Let $\phi$ be a regular quadratic form over $F$ and $X=X_{\phi}$ its projective quadric. 
\medskip 

\noindent{\underline{\it Case 1:} $\phi$ is nonsingular.}

\medskip 

We write $\phi=[a,b] \perp \psi$. Let $Y=X_{\left<a\right>\perp \psi}$ and $U=X\setminus Y$. We have an exact sequence
\begin{eqnarray*}
\CH^1(Y)\stackrel{i_*}{\longrightarrow} \CH^2(X) \stackrel{j^*}{\longrightarrow} \CH^2(U)\longrightarrow 0
\label{exactseq1}
\end{eqnarray*}
induced by the inclusions  $j: U\hookrightarrow  X$ and $i: Y\hookrightarrow  X$. Since $\CH^1(Y)=\Z .h$, where $h$ is the hyperplane section of $Y$, we get
\begin{eqnarray}
\CH^2(X)/i_*(\CH^1(Y))\simeq \TCH^2(X)\simeq \CH^2(U).
\label{isom1}
\end{eqnarray}

Note that $U$ is the affine variety given by: $ax^2+x+b + \psi(y_1,\ldots, y_n)$. Let $\pi: U\longrightarrow \A^1$ be the morphism given by $(x,y_1, \ldots, y_n)\mapsto x$. Then we have an exact sequence $$\coprod_{\alpha \in (\A^1)^1} \CH^1(U_{\alpha})\longrightarrow \CH^2(U)\longrightarrow \CH^2(U_{\theta})\longrightarrow 0,$$where $U_{\alpha}$ is the fiber of $\pi$ over the closed point $\alpha$, and $U_{\theta}$ is the generic fiber. Note that $U_{\theta}$ is the affine quadric over the rational function field $F(t_1)$ given by: $at_1^2+t_1+b +\psi$, and $U_{\alpha}$ is defined over the residue field $F(\alpha)$ by the affine quadric: 
\[
\overline{at_1^2+t_1 +b} +\psi.
\] 

\noindent{\it Claim 1:} $\CH^1(U_{\alpha})=0$ for each $\alpha$, and thus $\CH^2(U)\simeq \CH^2(U_{\theta})$.

In fact, we have the exact sequence
\begin{equation}
\CH^0(U_2) \longrightarrow \CH^1(U_1) \longrightarrow \CH^1(U_{\alpha})\longrightarrow 0,
\label{eq}
\end{equation}
where $U_2$ is the projective quadric given by $\psi_{F(\alpha)}$, and $U_1$ is the projective quadric given by 
\[
\langle \overline{at_1^2+t_1 +b}\rangle \perp \psi_{F(\alpha)}.
\]
 Since $\CH^0(U_2)=\Z.[U_2]$, $U_2$ is of codimension $1$ in $U_1$ and $\CH^1(U_1)$ is generated by the hyperplane section, it follows from (\ref{eq}) that $\CH^1(U_{\alpha})=0$. 

\medskip 

\noindent{\it Claim 2:} $\CH^2(U_{\theta})\simeq \TCH^2(Z)$, where $Z$ is the projective quadric defined over $F(t_1)$ by the quadratic form $\left<at_1^2+t_1 + b\right> \perp \psi$. To this end, we see $U_{\theta}$ as $Z\setminus Z'$, where $Z'$ is the projective quadric given by $\psi_{F(t_1)}$, and we use the exact sequence
\begin{eqnarray*}
\CH^1(Z')\longrightarrow \CH^2(Z)\longrightarrow \CH^2(U_{\theta})\longrightarrow 0
\end{eqnarray*}to get the desired claim as we did for (\ref{isom1}).

Hence, combining (\ref{isom1}) with claims 1 and 2 yields $\TCH^2(X)\simeq \TCH^2(Z)$.

\medskip 

\noindent{\underline{\it Case 2:} $\phi$ is singular.}

\medskip 
We write $\phi = \left<a_1\right> \perp \psi$, where $\psi =[a_2, b_2]\perp \psi'$ is nonsingular. As at the beginning of the case 1, we have $\TCH^2(X)\simeq \CH^2(U)$, where $U$ is the affine quadric given by: $a_1 + \psi$. 

Let $U'$ be the affine quadric given by $a_1+ \left<a_2\right>\perp \psi'$, and let $U''=U\setminus U'$. Clearly, $U''$ is the affine variety given by: $\left<a_1\right> + a_2x_2^2+x_2+b_2 +\psi'$. We have the exact sequence:
\begin{eqnarray*}
\CH^1(U')\longrightarrow \CH^2(U)\longrightarrow \CH^2(U'')\longrightarrow 0.
\end{eqnarray*}

\noindent{\it Claim 1.} $\CH^1(U')=0$, and thus $\CH^2(U)\simeq \CH^2(U'')$.

To this end, and using the exact sequence 
$$
\CH^0(Z')\longrightarrow \CH^1(Z)\longrightarrow \CH^1(U')\longrightarrow 0,
$$
it suffices to prove that $\CH^1(Z)=\Z.h$, where $Z$ is the projective quadric given by $\psi'':=\left<a_1, a_2\right> \perp \psi'$, and $Z'$ is the projective quadric given by $\left<a_2\right>\perp \psi'$. 

In fact, let $K/F$ be a separable quadratic extension such that $\psi''$ is isotropic, and $G$ the Galois group of $K/F$. Using the same argument as in \cite[Subsection (2.1)]{Kar91} we have $Z_K\setminus Z''\simeq \A_K^d$, where $d=\dim Z$ and $Z''$ is a singular quadric of dimension $\dim Z-1$. Since $\CH^1(\A^d_K)=0$ \cite[]{Fultonbook}, it follows that $\CH^1(Z_K)\simeq \CH^0(Z'')=\Z.[Z'']$, and thus $\CH^1(Z_K)$ is generated by $h$. By the same argument used for the proof of \cite[Lem. 2.4]{Kar91} we have an injection $\CH^1(Z)\hookrightarrow  (\CH^1(Z_K))^G$, and consequently, $\CH^1(Z)=\Z.h$.

Now, let $\pi: U''\longrightarrow \A^1$ be the morphism defined by: $(x_1, x_2, x_3, y_3, \ldots, x_n, y_n)\mapsto x_2$. This induces the exact sequence
$$
\coprod_{\alpha \in (\A^1)^1}\CH^1(U''_{\alpha}) \longrightarrow \CH^2(U'') \longrightarrow \CH^2(U''_{t_2})\longrightarrow 0,
$$
where $U''_{t_2}$ the fiber over the generic point given by the affine quadric $\left<a_1\right> + a_2t_2^2+t_2+b_2 +\psi'$ over $F(t_2)$, and $U''_{\alpha}$ is the fiber over the closed point $\alpha$ defined by the affine quadric 
\[
\langle a_1\rangle + \overline{a_2t_2^2+t_2+b_2} +\psi'
\]
 over $F(\alpha)$. By the Claim 1 before, we have $\CH^1(U''_{\alpha})=0$ for each $\alpha$. Then, $\CH^2(U'')\simeq \CH^2(U''_{t_2})$.

\medskip

Now let $V=U''_{t_2}$ and $\pi: V\longrightarrow \A^1$ be the morphism defined by: $(x_1, x_3, y_3, \ldots, x_n, y_n)\mapsto x_1$. Then, again we have the exact sequence
$$
\coprod_{\alpha \in (\A^1)^1} \CH^1(V_{\alpha})\longrightarrow \CH^2(V)\longrightarrow \CH^2(V_{t_1})\longrightarrow 0,
$$
where $V_{\alpha}$ is the fiber over the closed point $\alpha$, and $V_{t_1}$ is the generic fiber. Note that $V_{t_1}$ is the affine quadric over the rational function field $F(t_2)(t_1)$ given by: 
\[
a_2t_2^2+t_2+b_2 +  a_1t_1^2 + \psi'(x_3,y_3, \ldots, x_n, y_n),
\]
 and $V_{\alpha}$ is defined over the residue field $F(t_2)(\alpha)$ by the affine quadric: 
 \[
 \overline{a_2t_2^2+t_2+b_2 +  a_1t_1^2} +\psi'.
 \]
For each point $\alpha$, we have $\CH^1(V_{\alpha})=0$ as we did in the case 1. Hence, $\CH^2(V)\simeq \CH^2(V_{t_1})$, and thus $\TCH^2(X)\simeq \CH^2(V_{t_1})$. 

Using the same arguments as in Claim 2 of the Case 1, we deduce that $\TCH^2(X)\simeq \TCH^2(W)$, where $W$ is the projective quadric given by $\langle a_1t_1^2+a_2t_2^2+t_2+b_2\rangle \perp \psi'$ over $F(t_1, t_2)$. 

Now repeating Case 1 and 2 it is clear that we get the desired result.
\medskip

\noindent{\bf Step 2:} Our aim in this step is to reduce the study of $\TCH^2(X)$, where $X$ is a quadric given by a regular form of dimension $9$, to the study of $\TCH^2(X_{\psi})$ such that $\psi$ is a $10$-dimensional quadratic form of trivial Arf and Clifford invariants. All the material needed for this reduction is explained in the paper by Karpenko \cite[Section 4]{Kar95} and remains true in characteristic two. So we will just give a brief idea on how to proceed. The main tool used by Karpenko is the Grothendieck group $K_0(X)$ of the quadric $X$. This group is equipped with the topological filtration: 
\[
\cdots \supset K_0(X)^{(p)} \supset K_0(X)^{(p+1)} \supset \cdots.
 \]
For any integer $p$, let $K_0(X)^{(p/p+1)}$ denote the quotient $K_0(X)^{(p)}/K_0(X)^{(p+1)}$. There is a connection between the Chow group and the Grothendieck groups given by an epimorphism 
\[
\CH^p(X) \longrightarrow K_0(X)^{(p/p+1)}, 
\]
defined by  $[Y]\mapsto [{\mathcal O}_Y]$. For our case $p=2$, this morphism is an isomorphism, see the explanations given in \cite[Subsection 3.1]{Kar91}. Moreover, in \cite[Section 4]{Kar91} Karpenko studies the elementary part of $K_0(X)$, which is defined as the subgroup of $K_0(X)$ generated by all $h^p$, for $p\geq 0$, where $h$ is the hyperplane section of $X$. Two facts have been proved in \cite{Kar95}:
\medskip

\noindent{\bf (Fact 1)} Let $\psi$ be an odd-dimensional quadratic form and $\phi=\psi \perp \left<-\det_{\pm}\psi\right>$. If for some $p$ the groups $K_0(X_{\phi})^{(i/i+1)}$ are elementary for all $i\leq p$, then the same thing holds for the groups $K_0(X_{\psi})^{(i/i+1)}$ for all $i\leq p$ \cite[Cor. 4.5]{Kar91}. In our case we should take $\psi=\left<a\right>\perp \psi'$ for $\psi'$ nonsingular and $\phi =a[1, \Delta(\psi')] \perp \psi'$.

\medskip 

\noindent{\bf (Fact 2)} Let $\phi$ be a quadratic form and $E/F$ a field extension that splits $C_0(\phi)$. If for some $p$ the groups $K_0(X_{\phi_E})^{(i/i+1)}$ are elementary for all $i\leq p$, then the same thing holds for the groups $K_0(X_{\phi})^{(i/i+1)}$ for all $i\leq p$ \cite[Cor. 4.9]{Kar91}. In our case we take $\phi$ nonsingular.

\medskip

\noindent{\bf Step 3 (The conclusion):}  Let $\psi \in IF \otimes W_q(F)$ of dimension $10$ and trivial Clifford invariant, and let $X$ be its projective quadric. Suppose that $\psi$ is not hyperbolic. Then, $\psi$ is isotropic and $\dim \psi_{an}=8$. Let $Y$ be the projective quadric given by $\psi_{an}$. Using the same arguments as in \cite[Subsection 2.2]{Kar91}, we get $\CH^2(X)\simeq \CH^1(Y)$. Since $\CH^1(Y)=\Z .h$, the group $\CH^2(X)$ is elementary, i.e., $K_0(X)^{(2/3)}$ is elementary. As $K_0(X)^{(1/2)}$ is also elementary, it follows from (Fact 2) that the groups $K_0(X_{\phi})^{(i/i+1)}$ are elementary for any $i\leq 2$ and any quadratic form $\phi \in IF \otimes W_q(F)$ of dimension $10$. Consequently, we deduce from (Fact 1) that the groups $K_0(X_{\phi})^{(i/i+1)}$ are elementary for any $i\leq 2$ and any regular quadratic form $\phi$ of dimension $9$. Hence, for such a quadratic form $\CH^2(X_{\phi})$ is generated by $h^2$, and thus $\CH^2(X_{\phi})$ is torsion free. This completes the proof of Theorem \ref{thtor}.

\end{proof}

\end{appendix}

\section*{Acknowledgements}

The authors would like to thank Jean-Pierre Tignol for the discussions and the words of advice all along this project, and  Nikita Karpenko for his help with writing the appendix. The authors also thank the anonymous referees for the helpful remarks on the submitted version.

The first author would like to thank the third author and Universit\'e d'Artois for their support and hospitality (in 2017) while a part of the work for this paper was done. He gratefully acknowledges support from the FWO Odysseus Programme (project Explicit Methods in Quadratic Form Theory).

\bibliographystyle{abbrv}
\bibliography{bibfile}

\def\cprime{$'$}
\begin{thebibliography}{10}

\bibitem{Albert:1968}
A.~Albert.
\newblock {\em Structure of Algebras}, volume~24 of {\em Colloquium
  Publications}.
\newblock American Math. Soc., 1968.

\bibitem{AmitsurRowenTignol:1979}
S.~A. Amitsur, L.~H. Rowen, and J.-P. Tignol.
\newblock Division algebras of degree {$4$}\ and {$8$}\ with involution.
\newblock {\em Israel J. Math.}, 33(2):133--148, 1979.

\bibitem{Arason:1975}
J.~K. Arason.
\newblock Cohomologische invarianten quadratischer {F}ormen.
\newblock {\em J. Algebra}, 36(3):448--491, 1975.

\bibitem{Arason:1979}
J.~K. Arason.
\newblock Wittring und {G}aloiscohomologie bei {C}harakteristik {$2$}.
\newblock {\em J. Reine Angew. Math.}, 307/308:247--256, 1979.

\bibitem{AAB}
J.~K. Arason, R.~Aravire, and R.~Baeza.
\newblock On some invariants of fields of characteristic {$p>0$}.
\newblock {\em J. Algebra}, 311(2):714--735, 2007.

\bibitem{AravireBaeza:1989}
R.~Aravire and R.~Baeza.
\newblock The behavior of the {$\nu$}-invariant of a field of characteristic
  {$2$} under finite extensions.
\newblock {\em Rocky Mountain J. Math.}, 19(3):589--600, 1989.
\newblock Quadratic forms and real algebraic geometry (Corvallis, OR, 1986).

\bibitem{AravireBaeza:1992}
R.~Aravire and R.~Baeza.
\newblock Milnor's {$k$}-theory and quadratic forms over fields of
  characteristic two.
\newblock {\em Comm. Algebra}, 20(4):1087--1107, 1992.

\bibitem{Baeza}
R.~Baeza.
\newblock {\em Quadratic forms over semilocal rings}.
\newblock Lecture Notes in Mathematics, Vol. 655. Springer-Verlag, Berlin-New
  York, 1978.

\bibitem{Baeza2}
R.~Baeza.
\newblock The norm theorem for quadratic forms over a field of characteristic
  {$2$}.
\newblock {\em Comm. Algebra}, 18(5):1337--1348, 1990.

\bibitem{Barry}
D.~Barry.
\newblock Decomposable and indecomposable algebras of degree 8 and exponent 2.
\newblock {\em Math. Z.}, 276(3-4):1113--1132, 2014.

\bibitem{Barry16}
D.~Barry.
\newblock Power-central elements in tensor products of symbol algebras.
\newblock {\em Comm. Algebra}, 44(9):3767--3787, 2016.

\bibitem{BarChap15}
D.~Barry and A.~Chapman.
\newblock Square-central and {A}rtin--{S}chreier elements in division algebras.
\newblock {\em Arch. Math}, 104(6):513--521, 2015.

\bibitem{EKM}
R.~Elman, N.~Karpenko, and A.~Merkurjev.
\newblock {\em The algebraic and geometric theory of quadratic forms},
  volume~56 of {\em American Mathematical Society Colloquium Publications}.
\newblock American Mathematical Society, Providence, RI, 2008.

\bibitem{Fultonbook}
W.~Fulton.
\newblock {\em Intersection Theory}.
\newblock Developments in Mathematics. Springer-Verlag, Berlin, 1998.

\bibitem{GaribaldiMerkurjevSerre}
S.~Garibaldi, A.~Merkurjev, and J.-P. Serre.
\newblock {\em Cohomological invariants in {G}alois cohomology}, volume~28 of
  {\em University Lecture Series}.
\newblock American Mathematical Society, Providence, RI, 2003.

\bibitem{GaribaldiParimalaTignol}
S.~Garibaldi, R.~Parimala, and J.-P. Tignol.
\newblock Discriminant of symplectic involutions.
\newblock {\em Pure Appl. Math. Q.}, 5(1):349--374, 2009.

\bibitem{Izhboldin:2000}
O.~Izhboldin.
\newblock {$p$}-primary part of the {M}ilnor {$K$}-groups and {G}alois
  cohomologies of fields of characteristic {$p$}.
\newblock In {\em Invitation to higher local fields ({M}\"unster, 1999)},
  volume~3 of {\em Geom. Topol. Monogr.}, pages 19--41. Geom. Topol. Publ.,
  Coventry, 2000.
\newblock With an appendix by Masato Kurihara and Ivan Fesenko.

\bibitem{Jacobson1996}
N.~Jacobson.
\newblock {\em Finite-dimensional division algebras over fields}.
\newblock Springer-Verlag, Berlin, 1996.

\bibitem{Kahn}
B.~Kahn.
\newblock Quelques remarques sur le {$u$}-invariant.
\newblock {\em S\'em. Th\'eor. Nombres Bordeaux (2)}, 2(1):155--161, 1990.

\bibitem{Kar91}
N.~A. Karpenko.
\newblock Algebro-geometric invariants of quadratic forms.
\newblock {\em Leningrad Math. J.}, {\bf 2}(1):119 -- 138, 1991.

\bibitem{Kar95}
N.~A. Karpenko.
\newblock Chow groups of quadrics and index reduction formula.
\newblock {\em Nova J. Algebra Geom.}, {\bf 3}(4):357--379, 1995.

\bibitem{Kar98}
N.~A. Karpenko.
\newblock Codimension 2 cycles on {S}everi{-}{B}rauer varieties.
\newblock {\em {$K$}-Theory}, {\bf 13}:305--330, 1998.

\bibitem{Kato:1982}
K.~Kato.
\newblock Symmetric bilinear forms, quadratic forms and {M}ilnor {$K$}-theory
  in characteristic two.
\newblock {\em Invent. Math.}, 66(3):493--510, 1982.

\bibitem{Knebusch}
M.~Knebusch.
\newblock Isometrien \"uber semilokalen {R}ingen.
\newblock {\em Math. Z.}, 108:255--268, 1969.

\bibitem{Laghribi:2011}
A.~Laghribi.
\newblock Les formes bilin\'eaires et quadratiques bonnes de hauteur 2 en
  caract\'eristique 2.
\newblock {\em Math. Z.}, 269(3-4):671--685, 2011.

\bibitem{MammoneShapiro:1989}
P.~Mammone and D.~B. Shapiro.
\newblock The {A}lbert quadratic form for an algebra of degree four.
\newblock {\em Proc. Amer. Math. Soc.}, 105(3):525--530, 1989.

\bibitem{MilnorHusemoller}
J.~Milnor and D.~Husemoller.
\newblock {\em Symmetric bilinear forms}.
\newblock Springer-Verlag, New York-Heidelberg, 1973.
\newblock Ergebnisse der Mathematik und ihrer Grenzgebiete, Band 73.

\bibitem{Rowen:1978}
L.~H. Rowen.
\newblock Central simple algebras.
\newblock {\em Israel J. Math.}, 29(2-3):285--301, 1978.

\bibitem{Wadsworth}
A.~R. Wadsworth.
\newblock Discriminants in characteristic two.
\newblock {\em Linear and Multilinear Algebra}, 17(3-4):235--263, 1985.

\end{thebibliography}
\end{document}